\documentclass{conm-p-l}

\usepackage{amsmath}
\usepackage{amsfonts}
\usepackage{amssymb}
\usepackage{graphicx}
\usepackage{comment}
\usepackage{xcolor}

\newtheorem{theorem}{Theorem}

\newtheorem{example}[theorem]{Example}

\newtheorem{lemma}[theorem]{Lemma}

\newtheorem{proposition}[theorem]{Proposition}
\newtheorem{remark}[theorem]{Remark}

\allowdisplaybreaks[1]

\newtheorem{thm}{Theorem}[section]

\theoremstyle{definition}

\numberwithin{equation}{section}

\newcommand{\p}{\partial}

\def\O{\Omega}
\newcommand{\C}{\mathbb C}

\newcommand{\R}{\mathbb R}
\newcommand{\Z}{\mathbb Z}

\def \l{\lambda}
\def \a{\alpha}

\newcommand{\e}{{\bf e}}

\newcommand{\G}{\Gamma}
\def \L{\Lambda}
\newcommand{\n}{\mathfrak n}

\def\p{\mathfrak p}

\def \j {\mathbf j}

\newcommand{\La}{\Lambda}

\renewcommand{\H}{\mathcal H}

\newcommand{\resumename}{R\'esum\'e}

\begin{document}

\title[Decomposition]{Decompositions of Generalized Wavelet Representations}


\author[B. Currey]{Bradley Currey}
\address{Dept.\ of Mathematics and Computer Science\\
St. Louis University\\
St. Louis, MO 63103 U.S.A.\\} 
\email{curreybn@slu.edu}

\author[A. Mayeli]{Azita Mayeli}
\address{Mathematics Department, Queensborough C. College,  City University of New
York, Bayside, NY 11362 U.S.A}
\email{amayeli@qcc.cuny.edu}

\author[V. Oussa]{Vignon Oussa}
\address{Dept.\ of Mathematics\\
Bridgewater State University\\
Bridgewater, MA  02324 U.S.A.\\} 
\email{Vignon.Oussa@bridgew.edu}

\subjclass[2000]{Primary }

\date{}

\begin{abstract}
Let $N$ be a simply connected, connected nilpotent Lie group which admits a
uniform subgroup $\Gamma.$ Let $\alpha$ be an automorphism of $N$ defined by
$\alpha\left(  \exp X\right)  =\exp AX.$ We assume that the linear action of
$A$ is diagonalizable and we do not assume that $N$ is commutative. Let $W$ be a
unitary wavelet representation of the semi-direct product
group $\left\langle \cup_{j\in\mathbb{Z}}\alpha^{j}\left(  \Gamma\right)  \right\rangle \rtimes\left\langle
\alpha\right\rangle $ defined by $W\left(  \gamma,1\right)  =f\left(
\gamma^{-1}x\right)  $ and $W\left(  1,\alpha\right)  =\left\vert \det
A\right\vert ^{1/2}f\left(  \alpha x\right)  .$ We obtain a decomposition of
$W$ into a direct integral of unitary representations. Moreover, we provide an
explicit unitary operator intertwining the representations, a precise
description of the representations occurring, the measure used in the direct
integral decomposition and the support of the measure. We also study the
irreducibility of the fiber representations occurring in the direct integral
decomposition in various settings. We prove that in the case where $A$ is an expansive automorphism then the decomposition of $W$ is in fact a direct integral of unitary irreducible representations each occurring with infinite multiplicities if and only if $N$ is non-commutative. This work naturally extends results obtained by H. Lim, J.
Packer and K. Taylor who obtained a direct integral decomposition of $W$ in
the case where $N$ is commutative and the matrix $A$ is expansive, i.e. all
eigenvalues have absolute values larger than one. 
\end{abstract}
\keywords{wavelets, nilpotent, Lie groups, Direct integral, decomposition}

\subjclass[2010]{22E25}
\maketitle


\section{Introduction} 
The main purpose of this paper is to present an abstract harmonic analysis approach to the theory of wavelets on both commutative and non commutative groups. Although in the past few years, there has been extensive work done to extend the concept of wavelets to non-commutative domains \cite{Fuhr cont,Currey,Currey1,CurreyMayeli,CurreyTom,Ous} this theory is not entirely well-understood. We recall that, in its classical definition, a wavelet system is an orthonormal basis generated by a combination of integral shifts and dyadic dilations of a
single function in $L^{2}\left(\mathbb{R}\right).$  Although the set of operators involved in a wavelet system does not form a group, it generates a group which is isomorphic to a subgroup of the one-dimensional affine group also known as the $ax+b$ group. To be more precise, let $
\Gamma_{2}=\left\{  m2^{n}\in\mathbb{Q}:m,n\in\mathbb{Z}\right\}$ and let $
\varphi:\mathbb{Z}\rightarrow\mathrm{Aut}\left(  \Gamma_{2}\right)$ be defined such that $
\varphi\left(  m\right)  \gamma=2^{-m}\gamma$ for $\gamma\in\Gamma_{2}$ and $m\in\mathbb{Z}.$  We define the unitary operators $
D,$ and $T$ such that 
\begin{align*}
Df\left(  t\right)    =\sqrt{2}f\left(  2t\right) ,\text{ and }
T_{\gamma}f\left(  t\right)     =f\left(  t-\gamma\right).
\end{align*}
An orthonormal wavelet is a unit vector $\psi\in L^{2}\left(
\mathbb{R}
\right)  $ such that
$
\left\{  D^{n}T_{k}\psi:n\in%
\mathbb{Z}
,k\in%
\mathbb{Z}
\right\}
$
forms an orthonormal basis for $L^{2}\left(
\mathbb{R}
\right)  .$ It is not too hard to see that the subgroup of $\mathcal{U}\left(
L^{2}\left(
\mathbb{R}
\right)  \right)  $ which is generated by the operators $D$ and $T$
is isomorphic to $\Gamma_{2}\rtimes_{\varphi}%
\mathbb{Z}
$ via the faithful representation%
\[
W:\left(  \gamma,n\right)  \longmapsto T_{\gamma}D^{n}.
\]
Moreover,  $\Gamma_{2}\rtimes_{\varphi}%
\mathbb{Z}
$ is a finitely generated metabelian solvable group generated by two elements and $\Gamma_2\rtimes_{\varphi}\Z$ has the following finite presentation: $$\langle a,m: mam^{-1}=a^{2} \rangle.$$  The representation $W$ of this group was termed the \textit{wavelet representation} by L-H. Lim, J. Packer and K. Taylor in \cite{Taylor}. In \cite{Taylor}, the authors obtained a direct integral decomposition of $W$ into its irreducible
components which are some monomial representations parametrized by a wavelet
set. More precisely, they show that $W$ is equivalent to
\[
\int_{E}^{\oplus}\mathrm{Ind}_{\Gamma_{2}}^{\Gamma_{2}\rtimes_{\varphi}%
\mathbb{Z}
}\left(  \chi_{t}\right)  \text{ }dt.
\]
The representations occurring in the direct integral decomposition are irreducible, and $E$ is a subset of $\mathbb{R}$ which up to a null set tiles the real line by both dyadic dilations and integral
translations. Their results were actually obtained for the more general case
where
\[
D,T_{k}:L^{2}\left(\mathbb{R}^{n}\right)  \rightarrow L^{2}\left(\mathbb{R}^{n}\right)
\]
such that
\[
Df\left(  t\right)  =\det\left(  A\right)  ^{1/2}f\left(  At\right)  ,\text{
}T_{k}f\left(  t\right)  =f\left(  t-k\right),
\]
$A$ is an expansive matrix (all eigenvalues have absolute value greater
than one) in $M\left(  n,%
\mathbb{Z}
\right)  \cap GL\left(  n,%
\mathbb{Q}
\right)  $ and $k\in%
\mathbb{Z}
^{n}.$

The present work undertakes a thorough investigation of the case where $%
\mathbb{R}
^{n}$ is replaced by a simply connected, connected nilpotent
Lie group $N$. In this work, we do not assume that $N$ is commutative. If $N$ is commutative, then $N=\R^n$ and the main purpose of this paper is to generalize the results obtained in \cite{Taylor}. However, in this more general setting, it is not straightforward to define a wavelet representation acting in $L^{2}\left(  N\right).$ In order to summarize our results, we must first take care of some technical issues.

Let $\n$ be an algebra of $m\times m$ nilpotent real matrices and set $N = \exp \n$. Then $\n$ is a nilpotent Lie subalgebra of $\mathfrak{gl}(m,\R)$, $N$ is a closed, simply connected nilpotent subgroup of $GL(m,\R)$ and the matrix exponential $\exp : \n \rightarrow N$ is a bijection. 
A basis $\{X_1, X_2, \dots , X_n\}$ for $\n$ can be chosen in such a way that $\n_i = \text{ span}_\R \{ X_1, X_2, \dots , X_i\}$ is an ideal in $\n$. Such a basis is called a strong Malcev basis, and from now on, we fix such a basis for $\n$. Thus $\n$ is identified with $\R^n$, and we let $dX$ denote the Lebesgue measure on $\n$. The mapping
$X \mapsto \exp X$ is a homeomorphism, as is the mapping
$$
(t_1, t_2 ,\dots , t_n) \mapsto \exp t_1 X_1 \exp t_2X_2 \cdots \exp t_nX_n.
$$
Both coordinate systems induce the same bi-invariant, Borel measure $\mu$ on $N$ as the push forward of Lebesgue measure. 

Let $N$ be a connected nilpotent Lie group. A discrete subgroup $\Gamma$ is a \textit{uniform subgroup} if $N/\Gamma$ is compact.  Since nilpotent Lie groups are unimodular, then it can be shown that $N/\Gamma$ has a finite volume if $N/\Gamma$ is compact.  Not every simply connected nilpotent Lie group admits a uniform subgroup. A necessary and sufficient condition for existence of a uniform subgroup is that there is a strong Malcev basis for which the associated structure constants are rational numbers. In this case such a basis can be chosen so that 
$$
\G =  \exp \Z X_1\exp \Z X_2 \cdot \cdots   \exp\Z X_n
$$
and henceforth we assume that this is the case. 


Each element $x \in N$ defines a unitary translation operator $T_x$ on $L^2(N)$ by
$$
(T_x f)(y) = f(x^{-1} y).
$$
In order to define dilations in $L^2(N)$, we fix a Lie algebra automorphism $A \in GL(\n)$. Via the chosen basis for $\n$, $A$ is given by an invertible $n\times n$ real matrix and we write evaluation of $A$ on elements of $\n$ multiplicatively. Now, we will assume that $A$ defines an automorphism $\a$ of $N$ by $
\a(\exp X) = \exp AX$ such that $$AX_k=a_k X_k, a_k \in \R. $$ It follows from our assumption that $A$ is a diagonal matrix with respect to the fixed strong Malcev basis for the Lie algebra $\n$. For any integrable function $f$ on $N$,
$$
\int_N f(\a x) |\det A| d\mu(x) = \int_\n f(\exp AX) |\det A| dX = \int_\n f(\exp X ) dX = \int_N f(x) d\mu(x).
$$
Define
the unitary operator $D$ on $L^2(N)$ by 
$$
Df(x) = |\det A |^{\frac{1}{2}}  \ f(\a x ).
$$
Let $\mathcal G(D, T, \G)$ denote the subgroup of the group $\mathcal U(L^2(N))$ of unitary operators on $L^2(N)$, generated by the operators $D$ and $ T_\gamma, \gamma \in \G$. Then  $\mathcal G(D, T,\G)$ is the image of a unitary representation of the group $G = \Gamma_\a \rtimes H$, where $H = \left\langle \alpha\right\rangle $ is the free abelian group
generated by $\alpha$, and $\G_\a$ is the subgroup of $N$ generated by 
$$
\cup_{m\in \Z} \a^m(\G).
$$
$G$ is a subgroup of the 
semi-direct product $F=N\rtimes H$ with operation
\[
\left(  x_{1},\alpha^{m_{1}}\right)  \left(  x_{2},\alpha^{m_{2}}\right)
=\left(  x_{1}\alpha^{-m_{1}}x_{2},\alpha^{m_{1}+m_{2}}\right), \ \ \left(  x_{1},\alpha^{m_{1}}\right)  ,\left(  x_{2},\alpha^{m_{2}}\right)
\in F.
\]
The mapping  $V:F\rightarrow \mathcal U\left(  L^{2}\left(N\right)  \right)  $ defined by
\[
V\left(  x,\alpha^{m}\right)  =T_{x}D_{\alpha}^{m}
\]
is a unitary representation of $F$ acting in $L^2(N)$, and 
the {\it wavelet representation} is the restriction of $V$
to $G$. Write $W=V|_{G}.$ 

 The present work is organized as follows. The second section contains some standard preliminary work. In the third section, we provide a direct integral decomposition of the representation $W$ into smaller components. We describe the unitary representations occurring, the measure used in the direct integral and the support of the measure as well. We show that $W$ is equivalent to a measurable field of unitary representations over a set which tiles the unitary dual of $N$ by dilation. In the fourth section, we deal with the irreducibility of the representations occurring in the direct integral decomposition. We prove that in the case where $A$ is an expansive automorphism then the decomposition of $W$ is in fact a direct integral of unitary irreducible representations, each occurring with infinite multiplicities if and only if $N$ is non-commutative. We also discuss some surprising results derived from Bekka and Driutti's work \cite{Bekka}. More precisely, we show that there are instances where the irreducibility of representations occurring in the decomposition of $W$ is completely independent of the nature of the dilation action coming from $H.$ That is, there are examples (for non-commutative $N$) where the irreducibility (or reducibility) of the representations occurring in the decomposition of $W$ only depends on the structure constants of the Lie algebra $\n.$ Several examples are presented throughout the paper to help the reader follow the stream of ideas.
 

\section{Preliminaries}
Recall that the Fourier transform $f \mapsto \hat f$ on $L^1(\R^n)$ satisfies
$$
\widehat{T_x f}(\l) = e^{2\pi i \langle x,\l \rangle} \hat f(\l)
$$
and the set $\hat \R^n = \{ x \mapsto e^{2\pi i \langle x,\l \rangle} : \l \in \R^n\} $ of exponentials is precisely the set of continuous unitary homomorphisms of $\R^n$ into $\C^*$. 
To define the group Fourier transform for the present class of simply connected nilpotent groups $N$, we denote by $\hat N$ the space of equivalent classes of strongly continuous, irreducible unitary representations of $N$, where $\pi_1 : N \rightarrow  \mathcal U(\mathcal H_1)$ and  $\pi_2 : N \rightarrow  \mathcal U(\mathcal H_2)$ are equivalent if there is a unitary operator $U : \mathcal H_1 \rightarrow \mathcal H_2$ such that $U \circ \pi_1(x) = \pi_2(x) \circ U$ holds for all $x \in N$.  

Given an irreducible unitary representation $\pi : N \rightarrow \mathcal U\left(\mathcal H\right)$ and $f \in L^1(N)$, we define the operator $\hat f(\pi)$ on $\mathcal H$ by
$$
\hat f(\pi) = \int_N f(x) \pi(x) d\mu(x)
$$
where the integral is taken in the weak sense.  One can show that $\hat f(\pi)$ is trace-class, and by the translation invariance of $\mu$, we have 
$$
\widehat{T_xf}(\pi) = \pi(x) \circ \hat f(\pi)
$$
holds for all $x \in N$.  

Just as the Fourier transform on $L^1(\R^n) \cap L^2(\R^n)$ extends to a unitary isomorphism of $L^2(\R^n)$ with $L^2(\hat\R^n)$, the group Fourier transform provides a similar isomorphism by which we study the wavelet representation on $L^2(N)$. Although the topology of $\hat \R^n$ is exactly the same at that of $\R^n$,  the topological structure of $\hat N$ is not even Hausdorff if $N$ is not commutative. In fact there is a canonical homeomorphism $\kappa : \n^* / \text{Ad}^*(N) \rightarrow \hat N$, where $\text{Ad}^* : N \rightarrow GL(\n^*)$ is the coadjoint representation of $N$ acting on the linear dual $\n^*$ of $\n$. In this case, it is natural to parametrize a conull subset of $\hat N$ by an explicit subset $\La$ of $\n^*$ which is a cross-section for almost all of the coadjoint orbits. We make the following somewhat technical digression in order to describe the explicit Plancherel isomorphism. See \cite{Corwin} for more details and the original sources.

For a set $X$ we denote by $|X|$ the number of elements in the set $X.$ Any coadjoint orbit has the structure of a symplectic manifold, and hence is even dimensional. Let $2d $ be the maximal dimension of the coadjoint orbits. One has the following:

\begin{itemize}

\item[a)]subsets $\j \subset \e \subset  \{ 1, 2, \dots , n\}$ such that $| \e | = 2 | \mathbf{j} | = 2d$, $$\j = \{j_1<j_2<\cdots <j_d\}$$

\item[b)]  an $\text{Ad}^*(N)$-invariant Zariski open subset $\O$ whose orbits have dimension $2d$ and such that $\Lambda=M\cap \O$ is a cross-section for the coadjoint orbits in $\O$, where $$M = \{ \ell \in \n^* : \ell(X_j) = 0,\text{  for all } j \in \e\}$$ 

\item[c)] for each $\l \in \La$, an analytic subgroup $P(\l)$ of $N$ such that

\begin{itemize}

\item[c1)] $(t_1,t_2, \dots , t_d) \mapsto \exp t_1X_{j_1} \exp t_2X_{j_2} \cdots \exp t_dX_{j_d} \cdot  P(\l)$
is a homeomorphism of $\R^d$ with $N / P(\l)$, 

\item[c2)] $\chi_\l(\exp Y) = e^{2\pi i\l(Y)}$ defines a unitary character of $P(\l)$,

\item[c3)] the unitary representation $\pi_\l$ of $N$ induced from $P(\l)$ by $\chi_\l$ is irreducible and associated with $\l$ under the canonical mapping $\kappa$.

\end{itemize}

\end{itemize}

\noindent
The subalgebra $\p(\l) = \log(P(\l))$ is defined by
\begin{equation}\label{polar}
\p(\l) = \sum_{i = 1}^n \n_i(\l)
\end{equation}
where \begin{equation}\label{radical} \n_i(\l) = \{ Y \in \n_i : \l [X,Y] = 0\text{ for all } X\in \n_i\}.\end{equation} 
Now, let $q:N\to N/P(\lambda)$ be the canonical quotient map. 
For each $\lambda \in \La,$ the Hilbert space for the induced representation $\pi_\l$ is the completion $\mathcal H_\l$ of the space of complex-valued functions $g$ on $N$ satisfying the following:
\begin{enumerate}
\item the image of the support of $g$ through the quotient map $q$ is compact
\item $g(xp) = \chi_\l(p)^{-1} g(x), \ \ x\in N, p \in P(\l)$
\item $\int_{\R^d} | g(\exp t_1X_{j_1} \exp t_2X_{j_2} \cdots \exp t_dX_{j_d} ) |^2 dt < \infty.$
\end{enumerate}

The induced representation $\pi_\l$ is just the action of left translation: $\pi_\l(y)g(x) = g(y^{-1} x)$, and the mapping $U : \mathcal H_\l \rightarrow L^2(\R^d)$ defined by 
$$
Ug(t_1, t_2, \dots , t_d) = g(\exp t_1X_{j_1} \exp t_2X_{j_2} \cdots \exp t_dX_{j_d} ) 
$$
is an isomorphism of Hilbert spaces.

\begin{remark}
If $N $ is abelian, then  $\mathbf{e} = \emptyset$, $d = 0$, and $P(\l) = N$ for all $\l$. In this case the dual $\hat N$ of $N = \R^n$ is just the set $\hat \R^n$ defined above. 
\end{remark}

For $f\in L^{1}\left(  N\right)  \cap L^{2}\left(  N\right),$ we put $
\widehat{f}\left(  \lambda\right)  =\widehat{f}\left(  \pi_{\lambda}\right)$ 
and the map $f\mapsto\widehat{f}$ extends to a unitary operator
\[
\mathcal{P}:L^{2}\left(  N\right)  \rightarrow\int_{\Lambda}^{\oplus
}\mathcal{HS}\left(  L^{2}\left(\mathbb{R}^{d}\right)  \right)  \text{ }\left\vert \mathbf{P}\left(  \lambda\right)
\right\vert d\lambda
\]
where $\mathbf{P}$ is a non-vanishing polynomial function on $\Lambda$ and $\mathcal{HS}=L^2(\R^d)\otimes L^2(\R^d).$ Next,
we consider the map
\[
\mathcal{P}_{1}:L^{2}\left(  N\right)  \rightarrow\int_{\Lambda}^{\oplus
}\mathcal{HS}\left(  L^{2}\left(\mathbb{R}^{d}\right)  \right)  \text{ }d\lambda
\]
such that
\[
\mathcal{P}_{1}f=\left(  \widehat{f}\left(  \lambda\right)  \sqrt{\left\vert
\mathbf{P}\left(  \lambda\right)  \right\vert }\right)  _{\lambda\in\Lambda}.
\]
Then it is also clear that the map $\mathcal{P}_{1}$ is a unitary map obtained
by modifying the Plancherel transform $\mathcal{P}.$ \vskip 0.5cm

We shall now present a few examples of nilpotent Lie groups, their duals and the associated Plancherel measures.

\begin{example}\label{Heisenberg}
\end{example} \noindent
The Heisenberg group $N$ has as a basis for its Lie algebra $\{Z, Y, X\}$ where $[X,Y] = Z$ and all other brackets vanish. Putting $X_1 = Z, X_2 = Y, X_3 = X$, we see that the structure constants for this basis are rational and hence $\n$ has a rational structure. If we identify $N$ with $\R^3$ via coordinates  $n = (z,y,x) = \exp z Z \exp y Y \exp xX$, then the set  $\Gamma$ of integer points in $N$ is a uniform discrete subgroup of $N$ and the above basis is strongly based in $\Gamma$. Now, in this case $\e = \{2, 3\}$, and $\j = \{3\}$. Explicitly, 
$$
\O = \{ \ell \in \n^* : \ell(Z) \ne 0\}
$$
and 
$$
\La = \{ \lambda Z^* : \lambda \ne 0\} \simeq \R \setminus \{0\}.
$$
Note also that  $|\mathbf{P}(\ell )| = |\ell(Z)|$ in this example.

\begin{example}
\end{example} \noindent
Let $N$ be the upper triangular group of $4\times 4$ matrices. A typical
element of the Lie algebra $\mathfrak{n}$ is of the form
\[
\left[
\begin{array}
[c]{cccc}%
0 & u_{1} & y_{1} & z\\
0 & 0 & u_{2} & y_{2}\\
0 & 0 & 0 & u_{3}\\
0 & 0 & 0 & 0
\end{array}
\right]  .
\]
In fact, $\mathfrak{n}$ is spanned by the basis
\[
\left\{  Z,Y_{1},Y_{2},U_{1},U_{2},U_{3}\right\}
\]
with the following non-trivial Lie brackets%
\begin{align*}
\left[  U_{1},U_{2}\right]    & =Y_{1},\left[  U_{3},U_{2}\right]  =-Y_{2}\\
\left[  U_{1},Y_{2}\right]    & =Z,\left[  U_{3},Y_{1}\right]  =-Z.
\end{align*}
Put
\[
X_{1}=Z,X_{2}=Y_{1},X_{3}=Y_{2},X_{4}=U_{1},X_{5}=U_{2},X_{6}=U_{3}.
\]
We define
\[
\Gamma=\left\{  \left[
\begin{array}
[c]{cccc}%
1 & k_{4} & k_{2} & k_{1}\\
0 & 1 & k_{5} & k_{3}\\
0 & 0 & 1 & k_{6}\\
0 & 0 & 0 & 1
\end{array}
\right]  :k_{i}\in\mathbb{Z}\right\}.
\]
It is easy to see that $\Gamma$ is a discrete uniform subgroup of $N$ and the basis given above is strongly based in $\Gamma.$ Next,  let $\mathbf{e}=\left\{  2,3,4,6\right\} .$ Then
\[
\Omega=\left\{  \lambda\in\mathfrak{n}^{\ast}%
:\lambda\left(  Z\right)  \neq0\right\}
\]
and
\[
\Lambda=\left\{  \lambda\in\Omega:\lambda\left(  Y_{1}\right)  =\lambda\left(
Y_{2}\right)  =\lambda\left(  X_{1}\right)  =\lambda\left(  X_{3}\right)
=0\right\}  \simeq\mathbb{R}^{\ast}\times\mathbb{R}.
\]
Finally, the Plancherel measure is given by $
\left\vert \lambda_{1}\right\vert ^{2}d\lambda_{1}d\lambda_{2}$ on $\mathbb{R}^{\ast}\times\mathbb{R}.$

\begin{example}
\end{example} \noindent
Let $N$ be a nilpotent Lie subgroup of $GL\left(  10,\mathbb{R}\right)  $ such that a typical element of $N$ is of the form
\[
\left[
\begin{array}
[c]{cccccccccc}%
1 & 0 & 0 & x_{1} & x_{2} & x_{3} & -y_{1} & 0 & -y_{2}-y_{3} & z_{1}\\
0 & 1 & 0 & x_{2} & x_{3} & x_{2} & -y_{2} & -y_{1}-y_{3} & 0 & z_{2}\\
0 & 0 & 1 & x_{3} & x_{1} & x_{1} & -y_{3} & -y_{2} & -y_{1} & z_{3}\\
0 & 0 & 0 & 1 & 0 & 0 & 0 & 0 & 0 & \frac{1}{2}y_{1}\\
0 & 0 & 0 & 0 & 1 & 0 & 0 & 0 & 0 & \frac{1}{2}y_{2}\\
0 & 0 & 0 & 0 & 0 & 1 & 0 & 0 & 0 & \frac{1}{2}y_{3}\\
0 & 0 & 0 & 0 & 0 & 0 & 1 & 0 & 0 & \frac{1}{2}x_{1}\\
0 & 0 & 0 & 0 & 0 & 0 & 0 & 1 & 0 & \frac{1}{2}x_{2}\\
0 & 0 & 0 & 0 & 0 & 0 & 0 & 0 & 1 & \frac{1}{2}x_{3}\\
0 & 0 & 0 & 0 & 0 & 0 & 0 & 0 & 0 & 1
\end{array}
\right]  .
\]
The Lie algebra of $N$ is spanned by
\[
\left\{  Z_{1},Z_{2},Z_{3},Y_{1},Y_{2},Y_{3},X_{1},X_{2},X_{3}\right\}
\]
with non-trivial Lie brackets
\begin{align*}
\left[  X_{1},Y_{1}\right]   &  =Z_{1},\left[  X_{1},Y_{2}\right]
=Z_{2},\left[  X_{1},Y_{3}\right]  =Z_{3}\\
\left[  X_{2},Y_{1}\right]   &  =Z_{2},\left[  X_{2},Y_{2}\right]
=Z_{3},\left[  X_{2},Y_{3}\right]  =Z_{2}\\
\left[  X_{3},Y_{1}\right]   &  =Z_{3},\left[  X_{3},Y_{2}\right]
=Z_{1},\left[  X_{3},Y_{3}\right]  =Z_{1}.
\end{align*}
Therefore
\[
\Lambda=\left\{
\begin{array}
[c]{c}%
\lambda\in\mathfrak{n}^{\ast}:\det\left[
\begin{array}
[c]{ccc}%
\lambda\left(  Z_{1}\right)  & \lambda\left(  Z_{2}\right)  & \lambda\left(
Z_{3}\right) \\
\lambda\left(  Z_{2}\right)  & \lambda\left(  Z_{3}\right)  & \lambda\left(
Z_{2}\right) \\
\lambda\left(  Z_{3}\right)  & \lambda\left(  Z_{1}\right)  & \lambda\left(
Z_{1}\right)
\end{array}
\right]  \neq0,\lambda\left(  X_{i}\right)  =\lambda\left(  Y_{i}\right)  =0\\
1\leq i\leq3
\end{array}
\right\}
\]
and the Plancherel measure is equivalent to
\[
\left\vert \left(  \lambda_{1}-\lambda_{3}\right)  \left(  \lambda_{2}%
-\lambda_{3}\right)  \left(  \lambda_{1}+\lambda_{2}+\lambda_{3}\right)
\right\vert d\lambda_{1}d\lambda_{2}d\lambda_{3}
\]
defined over a Zariski open subset of $\mathbb{R}^{3}.$ 
For each $\lambda \in \Lambda$ the corresponding irreducible representation $\pi_{\lambda}$ is realized as acting on $L^{2}\left(\mathbb{R}^{3}\right) $ as follows (see \cite{Oussa2} for more details)
$$
\pi_{\lambda}\left(  \exp\left(  t_{1}X_{1}+t_{2}X_{2}+t_{3}X_{3}\right)
\right)  f\left(  x_{1},x_{2},x_{3}\right)   =f\left(  x_{1}-t_{1}%
,x_{2}-t_{2},x_{3}-t_{3}\right) ,$$
\begin{align*}
&\pi_{\lambda}\left(  \exp\left(  t_{1}Y_{1}+t_{2}Y_{2}+t_{3}Y_{3}\right)
\right)  f\left(  x_{1},x_{2},x_{3}\right)\\ &=e^{\left(  -2\pi
i\left\langle \left[
\begin{array}
[c]{c}%
\lambda_{1}t_{1}+\lambda_{2}t_{2}+\lambda_{3}t_{3}\\
\lambda_{2}t_{1}+\lambda_{2}t_{3}+\lambda_{3}t_{2}\\
\lambda_{1}t_{2}+\lambda_{1}t_{3}+\lambda_{3}t_{1}%
\end{array}
\right]  ,\left[
\begin{array}
[c]{c}%
x_{1}\\
x_{2}\\
x_{3}%
\end{array}
\right]  \right\rangle \right)  }f(x_1,x_2,x_3)  ,
\end{align*}
and
\begin{align*}
&\pi_{\lambda}\left(  \exp\left(  t_{1}Z_{1}+t_{2}Z_{2}+t_{3}Z_{3}\right)
\right)  f\left(  x_{1},x_{2},x_{3}\right)\\ & =e^{\left(  2\pi i\left\langle
\left[
\begin{array}
[c]{c}%
\lambda_{1}\\
\lambda_{2}\\
\lambda_{3}%
\end{array}
\right]  ,\left[
\begin{array}
[c]{c}%
t_{1}\\
t_{2}\\
t_{3}%
\end{array}
\right]  \right\rangle \right)  }f\left(  x_{1},x_{2},x_{3}\right).
\end{align*}
To see more examples, we refer the reader to the book by Corwin and Greenleaf \cite{Corwin} which contains several other explicit examples.


\section{Direct Integral Decompositions}
In order to define a unitary dilation $\hat D$ on the Fourier transform side, we consider the action of the dilation group on $\hat N$. Let $\pi$ be an irreducible representation of $N$ and let $\a^m \in H$. Define $\a^m \cdot \pi$ by
$$
\a^m \cdot \pi = \pi \circ \a^m.
$$
Then $\a^m \cdot \pi$ is irreducible, and may or may not be isomorphic with $\pi$. Thus we have an action of $H$ on $\hat N$ that may well be non-trivial. 

At the same time $H$ acts naturally on $\n^*$ by $\alpha^m \cdot \lambda (X) = \lambda(A^mX)$, and we write $\alpha^{m}\cdot\lambda=A^{m}\lambda.$ Furthermore, since it is assumed that $A$ is a diagonal matrix with respect to the fixed strong Malcev basis of the Lie algebra $\n$, then it is clear that $\L$ is $H$-invariant, and we claim that the parametrization $\l \mapsto \pi_\l$ is $H$-equivariant with respect to the action on $\hat N$ and $\La$. To see this, we observe that the definition (\ref{polar}) shows that
$$
P(A\l) = \a^{-1}(P(\l)).
$$
Next, we define a unitary representation $C$ of the group $H$ such that $C:H\to \mathcal{U}(\H_\l)$ and given  $g \in \H_\l$,  $C(\a)g(x) = g(\a(x)) |\det A |^{1/2}.$ Thus, for $p \in P(A\l), x \in N$, $C(\a)g(xp) = \chi_{A\l}(p)^{-1} C(\a)g(x)$, and one easily checks that for each $\a\in H,$ the map $C(\a) : \H_\l \rightarrow \H_{A\l}$ is a unitary isomorphism.  Moreover, for each $m \in \Z$, 
\begin{equation}
C (\a^m) \pi_\l (\cdot)= \pi_{A^m\l} (\cdot) C(\a^m).
\end{equation}
Finally, identifying the Hilbert space $H_\l$ with $L^2(\R^d),$ given   $g\in L^2(\R^d)$, we have the following:
$$
C(\a)g(t_1, t_2, \dots , t_d) = g(a_{j_1}t_1, a_{j_2}t_2, \dots , a_{j_d}t_d) \  |a_{j_1} a_{j_2} \cdots a_{j_d}|^{1/2}.
$$

\vskip 0.5cm
Let $\mathbb{I}$ be a countable set which is parameterizing an orthornormal basis for $L^2(\R^d).$ Fix such orthonormal basis
$$\mathcal{B}=\left\{  b_{\kappa}:\kappa\in\mathbb{I}\right\}  $$ for
$L^{2}\left(\mathbb{R}^{d}\right).$  It is worth noticing that in the case where $N$ is abelian, then $d=0$ and it is understood that the representation $C$ is simply the one-dimensional trivial representation. Moreover, if $N$ is commutative then $L^2\left(\R^d\right)=\C$ and $\mathbb{I}$ is a singleton. 
We recall that $$AX_k=a_k X_k$$ for some real numbers $a_k.$  $H $ \textit{acts trivially} on $\Lambda$ if and only if $a_k=1$ for all $k\not \in \mathbf{e}. $
\noindent So, we say that $H$ \textit{acts non-trivially} on $\Lambda$ if and only if there exists some index $k\not\in\mathbf{e}$ such that $a_k$ is not equal to $1.$
\begin{remark}
 Let $G$ be a group acting on a set $X$. We say that the action is {\rm effective} if $gx = x$ for all $x$ in $X$ implies that $g$ is the identity in $G$. Therefore, if $H$ acts non-trivially on $\Lambda$ then it must be the case that $H$ acts effectively on the set $\Lambda.$ \end{remark}

A measurable subset $E$ of $\L$ is called a {\it dilation tiling}  of $\L$ if $A^j E \cap A^m E = \emptyset$ for $j \ne m$, and $\cup_{m \in \Z} A^m E$ is conull in $\L$. We have the following.

\begin{proposition} Suppose that $H$ acts non-trivially on $\Lambda$, and let $E$ be a dilation tiling of $\L$. Then 
\[
V\simeq\int_{E}^{\oplus}\oplus_{\kappa\in\mathbb{I}
}\mathrm{Ind}_{N}^{N\rtimes H}\left(  \pi_{\lambda}\right)  d\lambda.
\]
\end{proposition}

\begin{proof}
We define a representation $\widehat{V}$ of the group $F$ as $\widehat{V}\left(  \cdot\right)  =\mathcal{P}_{1}V\left(  \cdot\right)
\mathcal{P}_{1}^{-1}.$ Since $\mathcal{P}_{1}$ is a unitary map then clearly $\widehat{V}$ and $V$ are equivalent representations. Moreover, $\widehat{V}$ is acting in the Hilbert space $\int_{\Lambda}^{\oplus}L^2(\R^d)\otimes L^2(\R^d) d\lambda.$ This proof will be structured as follows. First, we will show that $L^2(N)$ can be decomposed into a direct sum of multiplicity-free spaces which are stable under the action of the representation $V.$ Second, we will obtain a decomposition of the restriction of $V$ on each multiplicity-free subspace. Fix $\kappa_{0}\in\mathbb{I}$. Let $v=v^{\kappa_{0}}:\Lambda\rightarrow
L^{2}\left(\mathbb{R}^{d}\right)  $ be a measurable vector-valued function defined such that
\[
v\left(  \lambda\right)  =\sum_{j\in\mathbb{Z}}C\left(  \alpha^{-j}\right)  b_{\kappa_{0}}1_{A^{-j}E}\left(  \lambda\right).
\]
We write $v\left(  \lambda\right)  =v_{\lambda}.$ Now, we define the multiplicity-free Hilbert space
\[
\mathcal{H}_{\kappa_{0}}=\left\{  \lambda\mapsto u_{\lambda}\otimes
v_{\lambda}^{\kappa_{0}}\in L^{2}\left(  \Lambda,\mathcal{HS}\left(
L^{2}\left(\mathbb{R}^{d}\right)  \right)  ,d\lambda\right)  :u_{\lambda}\in L^{2}\left(\mathbb{R}^{d}\right)  \right\}  .
\]
We would like to show that $\mathcal{H}_{\kappa_{0}}$ is $\widehat{V}
$-invariant space. Let $g\in\mathcal{H}_{\kappa_{0}}$ such that $g\left(
\lambda\right)  =u_{\lambda}\otimes v_{\lambda}^{\kappa_{0}}.$ First, it is
easy to see that
\[
\widehat{V}\left(  x\right)  g\left(  \lambda\right)  =\pi_{\lambda}\left(
x\right)  u_{\lambda}\otimes v_{\lambda}^{\kappa_{0}}\in L^{2}\left(\mathbb{R}^{d}\right)  \otimes v_{\lambda}^{\kappa_{0}}.
\]
Next, let $\delta$ be defined such that  $d(\alpha^{-m}x)=\delta(m) dx$. For ease of notation write $\widehat V(m) = \widehat V(\a^m), \a^m \in H$. Then
\begin{align*}
\widehat{V}\left(  m\right)  g\left(  \lambda\right)    & =C\left(  \alpha
^{m}\right)  \circ g\left(  A^{-m}\lambda\right)  \circ C\left(  \alpha
^{m}\right)  ^{-1}\delta\left(  m\right)  ^{1/2}\\
& =C\left(  \alpha^{m}\right)  u_{A^{-m}\lambda}\otimes C\left(  \alpha
^{m}\right)  v_{A^{-m}\lambda}^{\kappa_{0}}\delta\left(  m\right)  ^{1/2}\\
& =\delta\left(  m\right)  ^{1/2}C\left(  \alpha^{m}\right)  u_{A^{-m}\lambda
}\otimes C\left(  \alpha^{m}\right)  v_{A^{-m}\lambda}^{\kappa_{0}}.
\end{align*}
Since
\begin{align*}
v_{A^{-m}\lambda}^{\kappa_{0}}  & =\sum_{j\in\mathbb{Z}}C\left(  \alpha^{-j}\right)  b_{\kappa_{0}}1_{A^{-j}E}\left(  A^{-m}
\lambda\right)  \\
& =\sum_{j\in\mathbb{Z}}C\left(  \alpha^{-j}\right)  b_{\kappa_{0}}1_{A^{m-j}E}\left(  \lambda
\right)  \\
& =\sum_{s\in\mathbb{Z}}C\left(  \alpha^{-s-m}\right)  b_{\kappa_{0}}1_{A^{-s}E}\left(
\lambda\right)  \\
& =\sum_{s\in\mathbb{Z}}C\left(  \alpha^{-m}\right)  C\left(  \alpha^{-s}\right)  b_{\kappa_{0}
}1_{A^{-s}E}\left(  \lambda\right),
\end{align*}
then
\begin{align*}
\widehat{V}\left(  m\right)  g\left(  \lambda\right)    & =\delta\left(
m\right)  ^{1/2}C\left(  \alpha^{m}\right)  u_{A^{-m}\lambda}\otimes\left(\sum_{s\in%
\mathbb{Z}
}C\left(  \alpha^{m}\right)  C\left(  \alpha^{-m}\right)  C\left(  \alpha
^{-s}\right)  b_{\kappa_{0}}1_{A^{-s}E}\left(  \lambda\right)\right)  \\
& =\delta\left(  m\right)  ^{1/2}C\left(  \alpha^{m}\right)  u_{A^{-m}\lambda
}\otimes\left(\sum_{s\in\mathbb{Z}}C\left(  \alpha^{-s}\right)  b_{\kappa_{0}}1_{A^{-s}E}\left(  \lambda\right)\right)
\\
& =\delta\left(  m\right)  ^{1/2}C\left(  \alpha^{m}\right)  u_{A^{-m}\lambda
}\otimes v\left(  \lambda\right).
\end{align*}
Thus, $\widehat{V}\left(  m\right)  g\left(  \lambda\right)  \in L^{2}\left(\mathbb{R}^{d}\right)  \otimes v_{\lambda}^{\kappa_{0}}.$ This shows that indeed, $\mathcal{H}_{\kappa_{0}}$ is $\widehat{V}
$-invariant. Now, we define the unitary
map
\[
\Phi=\Phi_{\kappa_{0}}:\mathcal{H}_{\kappa_{0}}\rightarrow L^{2}\left(
E\times%
\mathbb{Z}
,L^{2}\left(
\mathbb{R}
^{d}\right)  ,d\lambda\right)  \simeq\int_{E}^{\oplus}l^{2}\left(
\mathbb{Z}
,L^{2}\left(
\mathbb{R}
^{d}\right)  \right)  d\lambda
\]
such that for $g\in\mathcal{H}_{\kappa_{0}},$ we write $g\left(
\lambda\right)  =u_{\lambda}^{g}\otimes v_{\lambda}^{\kappa_{0}}$ and
\[
\Phi g\left(  \lambda\right)  =\left(  C\left(  \alpha\right)  ^{-j}%
u_{A^{j}\lambda}^{g}\left\vert \det A\right\vert ^{j/2}\right)  _{j\in\mathbb{Z}
}.
\]
With some straightforward computations, we obtain
\[
\Phi\widehat{V}\left(  x\right)  g\left(  \lambda\right)  =\left(  C\left(
\alpha\right)  ^{-j}\pi_{A^{j}\lambda}\left(  x\right)  u_{A^{j}\lambda}%
^{g}\left\vert \det A\right\vert ^{j/2}\right)  _{j\in\mathbb{Z}}%
\]
and
\[
\Phi\widehat{V}\left(  m\right)  g\left(  \lambda\right)  =\left(  C\left(  \alpha\right)  ^{m-j}u_{A^{j-m}\lambda}%
^{g}\left\vert \det A\right\vert ^{\frac{j-m}{2}}\right)  _{j\in\mathbb{Z}}.
\]
Let $\rho_{\lambda}\simeq\mathrm{Ind}_{N}^{N\rtimes H}\left(  \pi_{\lambda
}\right)  $ be realized as acting in $l^{2}\left(\mathbb{Z},L^{2}\left(\mathbb{R}^{d}\right)  \right)  .$ Then,
\begin{align*}
\rho_{\lambda}\left(  x\right)  \Phi g\left(  \lambda\right)    & =\left(
\pi_{\lambda}\left(  \alpha^{j}x\right)  C\left(  \alpha\right)  ^{-j}%
u_{A^{j}\lambda}^{g}\left\vert \det A\right\vert ^{\frac{j}{2}}\right)  _{j\in%
\mathbb{Z}
}\\
& =\left(  C\left(  \alpha\right)  ^{-j}\pi_{A^{j}\lambda}\left(  x\right)
u_{A^{j}\lambda}^{g}\left\vert \det A\right\vert ^{\frac{j}{2}}\right)  _{j\in%
\mathbb{Z}
}\\
& =\Phi\widehat{V}\left(  x\right)  g\left(  \lambda\right)  .
\end{align*}
Similarly,  it is easy to see that
\begin{align*}
\rho_{\lambda}\left(  m\right)  \Phi g\left(  \lambda\right)    =\Phi\widehat{V}\left(  m\right)  g\left(  \lambda\right)  .
\end{align*}
Thus, the restriction of $\widehat{V}$ to the Hilbert space $\mathcal{H}%
_{\kappa_{0}}$ is equivalent to
\[
\int_{E}^{\oplus}\mathrm{Ind}_{N}^{N\rtimes H}\left(  \pi_{\lambda}\right)
d\lambda.
\]
Finally, we obtain
\[
V\simeq\widehat{V}\simeq\int_{E}^{\oplus}\oplus_{\kappa\in\mathbb{I}%
}\mathrm{Ind}_{N}^{N\rtimes H}\left(  \pi_{\lambda}\right)  d\lambda.
\]
This concludes the proof. 
\end{proof}

\begin{lemma}
If the action of $H$ is trivial on $\Lambda,$ then $\left\vert \det
A\right\vert =1.$
\end{lemma}
\begin{proof}
Let $\mathfrak{v}=\mathbb{R}
$-span $\left\{  X_{i}:i\in\mathbf{e}\right\}$, and let $\lambda \in \Lambda.$ We endow $\mathfrak{v}$
with the bi-linear form $\omega$ defined by
$\omega\left(  X,Y\right)  =\lambda\left(  \left[  X,Y\right]  \right).$ This bi-linear form is non-degenerate, and the vector space $\mathfrak{v}$ together with the non-degenerate bilinear form $\omega$ has the structure of a symplectic vector space (see Lemma $27$ \cite{Ous}). Next, since $A$ is a diagonal matrix then $\mathfrak{v}$ is $A$-invariant and the
restriction of $A$ to $\mathfrak{v}$ is a symplectic transformation. Therefore
$\left\vert \det A|_{\mathfrak{v}}\right\vert =1$ and $
\left\vert \det A\right\vert =\left\vert \det A|_{\mathfrak{v}}\right\vert
\left\vert \det A|_{\mathfrak{n}\ominus\mathfrak{v}}\right\vert =1.$
\end{proof}

Let $E_{1}\subseteq\mathbb{R}^{d}$ such that $\left\{  A_{1}^{m}E_{1}:m\in\mathbb{Z}\right\}  $ is a measurable partition of $\mathbb{R}^{d}$ and $A_{1}$ is the restriction of $A$ to the vector space $\mathbb{R}$-$\mathrm{span} \left\{  X_{j_{1}},\cdots,X_{j_{d}}\right\} .$ Fix an orthonormal basis for $L^2(E_1).$ More precisely,  $ \left\{b_j: j\in\mathbb{J}\right\}$ is a fixed orthonormal basis for  $L^2(E_1)$ and the set $\mathbb{J}$ is a parametrizing set for this orthonormal basis. The set $\mathbb{J}$ will be important for the following proposition. 
\begin{proposition}
Assume that $N$ is not commutative and that $H$ acts trivially on $\Lambda.$ Then
\[
V\simeq\int_{\Lambda}^{\oplus}\int_{\mathbb{T}}^{\oplus}\oplus_{\kappa
\in\mathbb{J}}\widetilde{\pi}_{\lambda,\sigma}\text{ }d\sigma d\lambda
\]
where $\widetilde{\pi}_{\lambda,\sigma}(x)=\pi_{\lambda}(x)$ for $x\in\Gamma_{\alpha}$  and $\widetilde{\pi}_{\lambda,\sigma}(\alpha)=C(\alpha)\chi_{\sigma}(\alpha)$ and $\chi_{\sigma}$ is a character of $H.$
\end{proposition}

\begin{proof}  We aim to construct an intertwining unitary operator for the representations described above. We recall that the representation $V$ is equivalent to $\widehat{V}$ which is
acting in $
\int_{\Lambda}^{\oplus}L^{2}\left(\mathbb{R}^{d}\right)  \otimes L^{2}\left(\mathbb{R}^{d}\right)  \text{ }d\lambda.$ To simplify our proof, we will only consider rank-one operators; since their linear span is dense in the Hilbert space of Hilbert-Schmidt  operators. We recall that $\widehat
{V}=\mathcal{P}_{1}\circ V\circ\mathcal{P}_{1}^{-1}$ and we write $\widehat
{V}=\int_{\Lambda}^{\oplus}\widehat{V}_{\lambda}\text{ }d\lambda$ such that
$$\widehat{V}_{\lambda}\left(  x\right)  g\left(  \lambda\right)  =\pi
_{\lambda}\left(  x\right)  u_{\lambda}^{g}\otimes v_{\lambda}$$
for $g\left(  \lambda\right)  =u_{\lambda}^{g}\otimes v_{\lambda}.$ It is also
fairly easy to see that $\widehat{V}_{\lambda}\left(  \alpha^{m}\right)
g\left(  \lambda\right)  =C\left(  \alpha^{m}\right)  \circ g\left(
\lambda\right)  \circ C\left(  \alpha^{m}\right)  ^{-1}.$  We define a unitary map $P_{1}:L^{2}\left(\mathbb{R}^{d}\right)  \rightarrow\int_{E_{1}}^{\oplus}l^{2}\left(\mathbb{Z}\right)  $ $dt$ where
\[
P_{1}f\left(  t,m\right)  =\left(  f\left(  A_{1}^{m}t\right)  \left\vert \det
A_{1}\right\vert ^{m/2}\right)  _{m\in\mathbb{Z}}.
\]
We also define a unitary map: $P_{2}:\int_{E_{1}}^{\oplus}l^{2}\left(\mathbb{Z}\right)  $ $dt\rightarrow\int_{E_{1}}^{\oplus}\int_{\mathbb{T}}^{\oplus}\mathbb{C}
$ $d\sigma dt$ such that
\[
P_{2}\left(  \left\{  \left(  a_{k}^{t}\right)  _{k\in\mathbb{Z}}\right\}  _{t\in E_{1}}\right)  \left(  s,\sigma\right)  =\sum_{k\in\mathbb{Z}}a_{k}^{s}e^{2\pi ik\sigma}\text{ (\textit{Fourier transform} of }\left(  a_{k}^{s}\right)  _{k\in\mathbb{Z}}.)
\]
Now, we define the unitary map
\[
Q:L^{2}\left(  \Lambda,\mathcal{HS}\left(  L^{2}\left(\mathbb{R}^{d}\right)  \right)  ,d\lambda\right)  \xrightarrow{\hspace*{1cm}}\int_{\Lambda}^{\oplus}
L^{2}\left(\mathbb{R}^{d}\right)  \otimes\left(\int_{E_{1}}^{\oplus}\int_{\mathbb{T}}^{\oplus}\mathbb{C}\text{ }d\sigma dt\right)d\lambda
\]
as follows:
\[
Q\left(  \left\{  u_{\lambda}\otimes v_{\lambda}\right\}  _{\lambda\in\Lambda
}\right)  =\left(  \left\{  u_{\lambda}\otimes P_{2}P_{1}\left(  v_{\lambda
}\right)  \right\}  _{\lambda\in\Lambda}\right).
\]
We write $Q=\int_{\Lambda}^{\oplus}Q_{\lambda}$ $d\lambda$ such that
\begin{align*}
Q_{\lambda}\widehat{V}_{\lambda}\left(  x\right)  \left(  u_{\lambda}\otimes
v_{\lambda}\right)    & =Q_{\lambda}\left(  \left(  \pi_{\lambda}\left(
x\right)  u_{\lambda}\right)  \otimes v_{\lambda}\right)  \\
& =\pi_{\lambda}\left(  x\right)  Q_{\lambda}\left(  u_{\lambda}\otimes
v_{\lambda}\right) 
\end{align*}
and
\begin{align*}
Q_{\lambda}\widehat{V}_{\lambda}\left(  m\right)  \left(  u_{\lambda}\otimes
v_{\lambda}\right)    & =Q_{\lambda}\left(  C\left(  \alpha^{m}\right)
u_{\lambda}\otimes C\left(  \alpha^{m}\right)  v_{\lambda}\right)  \\
& =C\left(  \alpha^{m}\right)  u_{\lambda}\otimes P_{2}P_{1}C\left(
\alpha^{m}\right)  v_{\lambda}\\
& =C\left(  \alpha^{m}\right)  u_{\lambda}\otimes\chi_{\sigma}\left(
m\right)  P_{2}P_{1}v_{\lambda}.
\end{align*}
The last equality above is justified by the following computations:
\begin{align*}
P_{2}P_{1}C\left(  \alpha^{m}\right)  v_{\lambda}\left(  t,\sigma\right)    &
=\sum_{k\in\mathbb{Z}}\left(  P_{1}C\left(  \alpha^{m}\right)  v_{\lambda}\right)  _{k}^{t}e^{2\pi
ik\sigma}\\
& =\sum_{k\in\mathbb{Z}}v_{\lambda}\left(  A_{1}^{k-m}t\right)  \left\vert \det A_{1}\right\vert
^{\frac{k-m}{2}}e^{2\pi ik\sigma}\\
& =\sum_{l\in\mathbb{Z}}v_{\lambda}\left(  A_{1}^{l}t\right)  \left\vert \det A_{1}\right\vert
^{\frac{l}{2}}e^{2\pi i\left(  l+m\right)  \sigma}\\
& =e^{2\pi im\sigma}\sum_{l\in\mathbb{Z}}v_{\lambda}\left(  A_{1}^{l}t\right)  \left\vert \det A_{1}\right\vert
^{\frac{l}{2}}e^{2\pi il\sigma}\\
& =e^{2\pi im\sigma}P_{2}P_{1}v_{\lambda}\left(  t,\sigma\right)  .
\end{align*}
Let $\left\{  b_{k}:k\in\mathbb{J}\right\}  $ be an orthonormal basis for
$L^{2}\left(  E_{1}\right)  .$ Define the unitary map
\[
\mathcal{Z}:\int_{E_{1}}^{\oplus}\int_{\mathbb{T}}^{\oplus}\text{ }\mathbb{C}\text{ }d\sigma dt\rightarrow\int_{\mathbb{T}}^{\oplus}\oplus_{k\in\mathbb{J}}\mathbb{C}\text{ }d\sigma
\]
such that
\[
\mathcal{Z}f\left(  \sigma\right)  =\left(  \left\langle f\left(  \cdot,\sigma
\right)  ,b_{k}\right\rangle \right)  _{k\in\mathbb{J}}.
\]
Via the map $\mathcal{Z},$ we identify the Hilbert space
\[
\int_{\Lambda}^{\oplus}L^{2}\left(\mathbb{R}^{d}\right)  \otimes\left(  \int_{E_{1}}^{\oplus}\int_{\mathbb{T}}^{\oplus
}\text{ }\mathbb{C}\text{ }d\sigma dt\right)  \text{ }d\lambda
\]
with
\[
\int_{\Lambda}^{\oplus}L^{2}\left(\mathbb{R}^{d}\right)  \otimes\left(  \int_{\mathbb{T}}^{\oplus}\oplus_{k\in\mathbb{J}}\mathbb{C}\text{ }d\sigma\right)  \text{ }d\lambda
\]
which is then identified with
\[
\int_{\Lambda}^{\oplus}\int_{\mathbb{T}}^{\oplus}L^{2}\left(\mathbb{R}^{d}\right)  \otimes\left(  \oplus_{k\in\mathbb{J}}\mathbb{C}\right)  \text{ }d\sigma d\lambda.
\]
Via the identifications described above, it follows immediately that the quasi-regular representation $V$ is unitarily equivalent to the representation
\[
\int_{\Lambda}^{\oplus}\int_{\mathbb{T}}^{\oplus}\oplus_{k\in\mathbb{J}}\widetilde{\pi}_{\lambda,\sigma
}\text{ }d\sigma d\lambda.
\]
This concludes the proof.
\end{proof}
\begin{proposition}
If $N=\mathbb{R}^{n}$ is commutative and $H$ acts trivially on $\Lambda$ then
\[
V\simeq\int_{\Lambda}^{\oplus}\widetilde{\pi}_{\lambda}\text{ }d\lambda
\]
where $\widetilde{\pi}_{\lambda}\left(  x,\alpha^{0}\right)  =e^{2\pi
i\left\langle x,\lambda\right\rangle }.$
\end{proposition}

\begin{proof}
Since $N$ is commutative, then all eigenvalues of the matrix $A$ are equal to
one if $H$ acts trivially on $\Lambda=\widehat{\mathbb{R}^{n}}$. Via the Plancherel transform, the representation $T$ of $N$ is unitarly
equivalent to $\int_{\widehat{\mathbb{R}}^{n}}^{\oplus}\chi_{\lambda}$ $d\lambda$. Since the representation $D$ of $H$ acts
trivially on $\Lambda$, then it follows that $
V\simeq\int_{\Lambda}^{\oplus}\widetilde{\pi}_{\lambda}\text{ }d\lambda.$
\end{proof}
 \

 The following lemma shows that a decomposition of $V$ yields immediately a decomposition of $W$. 

\

\begin{lemma} Let $N$ be a locally compact group, $H$ a group of automorphisms of $N$, and let $N_0$ be a subgroup of $N$ that is normalized by $H$. Let $\pi$ be a unitary representation of $N$. Then 
$$
\mathrm{Ind}_{N}^{N \rtimes H} (\pi) |_{N_0} \simeq \mathrm{Ind}_{N_0}^{N_0\rtimes H}(\pi |_{N_0}).
$$

\end{lemma} 

\begin{proof} Put $\tau = \text{Ind}_{N}^{N \rtimes H} (\pi) |_{N}$, and $\tau_0 = \text{Ind}_{N}^{N \rtimes H} (\pi) |_{N_0}$ which are acting in the Hilbert spaces $\mathcal H_\tau$ and $\mathcal H_{\tau_0}$ respectively. Since elements of $\mathcal H_\tau$ and $\mathcal H_{\tau_0}$ are both determined by their values on $\H$ , then the restriction map $g \mapsto g|_{N_0\rtimes H}$ is an isomorphism of $\mathcal H_\tau$ with $\mathcal H_{\tau_0}$. Given $g \in \mathcal H_\tau$, then for $x \in N_0, h \in H$ since $N_0$ is normalized by $H$,  $
\tau(xh) g(k) = g(h^{-1}x^{-1} k) \delta(h)^{-1/2} = g(h^{-1}k (k^{-1}x^{-1} k)) \delta(h)^{-1/2} = \pi(k^{-1}x k)g(h^{-1} k) = \tau_0(xh)g(k).$

\end{proof} 

\

The decomposition of the wavelet representation now follows from the preceding results.

\

\begin{thm}\label{main}

\

\noindent
1. Suppose that $H$ acts non-trivially on $\Lambda$, then

\[
W\simeq\int_{E}^{\oplus}\left(\oplus_{\kappa\in\mathbb{I}}\mathrm{Ind}%
_{\Gamma_{\alpha}}^{\Gamma_{\alpha}\rtimes H}\left(  \pi_{\lambda}%
|_{\Gamma_{\alpha}}\right)\right)  d\lambda.
\]

\

\noindent
2. If $H$ acts trivially on $\Lambda,$ and if $N$ is not commutative then the wavelet representation is
decomposed into a direct integral of representations as follows
\[
W\simeq\int_{\Lambda}^{\oplus}\int_{\mathbb{T}}^{\oplus}\left(\oplus_{\kappa
\in\mathbb{J}}\widetilde{\pi}_{\lambda,\sigma}|_{G}\right)d\sigma d\lambda.
\]

\

\noindent
3. If $H$ acts trivially on $\Lambda$ and if $N$ is commutative then \[
W\simeq\int_{\Lambda}^{\oplus}\widetilde{\pi}_{\lambda}|_{\Gamma\rtimes H} \text{ }d\lambda.
\]

\end{thm}

\vskip 1cm

We remark that the sets $\mathbb{I}$ and $\mathbb{J}$ are infinite sets if and only if $N$ is not commutative. 
Thus in both decompositions above, if $N$ is not commutative, the fiber representations are always decomposable into direct sums of equivalent representations. However, it is generally not true that the representations occurring in the direct sums are irreducible. In two examples below, we consider dilations on the three-dimensional Heisenberg group. For the first example, the representations occurring in the decomposition are direct sums of reducible representations. In the second example, the fiber representations occurring in the direct integral decompositions are direct sums of equivalent irreducible representations.  It then becomes obvious that sometimes the irreducibility of the representations occurring in the direct sums depends on some properties of the action of the automorphism $\alpha$ on $N.$ Surprisingly, we will see that there are also instances where the irreducibility of the representations occurring in the direct sums only depends on the Lie bracket structure of $\mathfrak{n}$. 

\begin{example}
\end{example} \noindent
Let $N$ be as in Example \ref{Heisenberg}: its Lie algebra is given by $\mathfrak{n}=\left(  X_{1},X_{2},X_{3}\right) _{\mathbb{R}}$ such that $\left[  X_{3},X_{2}\right]  =X_{1}$. Let $\alpha \in \mathrm{Aut}(N)$ be defined by $\alpha(\exp X_1)=\exp2X_1,\alpha(\exp X_2)=\exp 2X_2,\alpha(\exp X_3)=\exp X_3$.  As seen in Example \ref{Heisenberg}, we have 
$$
\L = \{ \l X_1^* : \l \ne 0\}
$$ 
which we identify with $\R \setminus \{0\}$, and for each $\l \in \L$, $A \l = 2\l$. If $$\Gamma =\exp\mathbb{Z}X_1 \exp\mathbb{Z}X_2 \exp\mathbb{Z}X_3$$ then 
$$
\Gamma_{\alpha}=\left\{  \exp\left(  2^{j}k_{1}X_{1}\right)  \exp\left(
2^{j}k_{2}X_{2}\right)  \exp\left(  k_{3}X_{3}\right)  :k_{1},k_{2},k_{3},j\in\mathbb{Z},j<0\right\}.$$

\

Let $E$ be a dilation tiling for $\L$ (for example, the Shannon set $(-1,-1/2] \cup [1/2,1)$.)  Then 
$$
W\simeq\int_{E }^{\oplus
}\oplus_{k\in\mathbb{I}}\text{ \textrm{Ind}}_{\Gamma_{\alpha}}^{\Gamma
_{\alpha}\rtimes H}\left(  \pi_{\lambda}|_{\Gamma_{\alpha}}\right)  \text{
}d\lambda.
$$
Moreover, the representation $\pi_{\lambda}|_{\Gamma_{\alpha}}$ acts on
$L^{2}\left(\mathbb{R}\right)  $ as follows:
\begin{align*}
\pi_{\lambda}\left(  \exp x_{1}X_{1}\right)  f\left(  t\right)    & =e^{2\pi
i\lambda x_{1}}f\left(  t\right)  \\
\pi_{\lambda}\left(  \exp x_{2}X_{2}\right)  f\left(  t\right)    & =e^{-2\pi
i\lambda x_{2}t}f\left(  t\right)  \\
\pi_{\lambda}\left(  \exp x_{3}X_{3}\right)  f\left(  t\right)    & =f\left(
t-x_{3}\right)  .
\end{align*}
Notice that \textrm{Ind}$_{\Gamma_{\alpha}}^{\Gamma_{\alpha}\rtimes H}\left(
\pi_{\lambda}|_{\Gamma_{\alpha}}\right)  $ is a reducible representation of
$\Gamma_{\alpha}$ since for $q<1$, the linear span of the set $
\left(  \pi_{\lambda
}|_{\Gamma_{\alpha}}\right)  \left(  \Gamma_{\alpha}\right)  \chi_{\left[
0,q\right]  }$ is not dense in $L^{2}\left(\mathbb{R}\right)  .$

It is easily seen that the closure of $\G_\a$ in $N$ is the group 
$$
N_0 = \{ \exp x_1 X_1 \exp x_2 X_2 \exp k X_3 : k \in \Z, x_1, x_2 \in \R\}.
$$
Put $P = P(\l) = \exp \left(\R X_1 + \R X_2\right)$. 
Since $\exp \Z X_3$ acts freely on $\hat P$ then  $\sigma_\l = \text{Ind}_P^{N_0} (\chi_\l)$ is irreducible by Mackey theory, as is $\sigma^t_\l:= \sigma_\l(\exp -tX_3 \cdot \exp t X_3), t \in \R$. Now by inducing in stages, $\pi_\l \simeq \text{Ind}_{N_0}^N(\sigma_\l)$, and a standard calculation shows then that 
$$
\pi_\l |_{N_0} \simeq \int_{[0,1)}^\oplus \ \sigma_\l^t \ dt.
$$
Since $\G_\a$ is dense in $N_0$, then $\sigma^t_\l |_{\G_\a}$ is also irreducible, so 
$$
\pi_\l |_{\G_\a} \simeq \int_{[0,1)}^\oplus \ \sigma_\l^t |_{\G_\a} \ dt
$$
is an irreducible decomposition of $\pi_\l |_{\G_\a}$. Finally, since $H$ acts freely on $\Gamma_\a$, then $ \text{ \textrm{Ind}}_{\Gamma_{\alpha}}^{\Gamma
_{\alpha}\rtimes H}(\sigma_\l^t |_{\G_\a}) $ is an irreducible representation of the wavelet group, and we obtain the irreducible decomposition
$$
W\simeq\int_{E }^{\oplus
}\oplus_{k\in\mathbb{I}}\text{ \textrm{Ind}}_{\Gamma_{\alpha}}^{\Gamma
_{\alpha}\rtimes H}\left(  \pi_{\lambda}|_{\Gamma_{\alpha}}\right)  \text{
}d\lambda \simeq \int_{E }^{\oplus
}\oplus_{k\in\mathbb{I}}\ \int_{[0,1)}^\oplus \ \text{ \textrm{Ind}}_{\Gamma_{\alpha}}^{\Gamma
_{\alpha}\rtimes H}(\sigma_\l^t |_{\G_\a}) dt \ d\l.
$$

\begin{example} \end{example}\noindent
Again let $N$ be the Heisenberg group with Lie algebra given as in the preceding. Now, let
$\alpha\in\mathrm{Aut}\left(  N\right)  $ be defined by
\begin{align*}
\alpha\left(  \exp X_1\right)    & =\exp X_1,\\
\alpha\left(  \exp X_2\right)    & =\exp\frac{X_2}{2},\\
\alpha\left(  \exp X_3\right)    & =\exp(2 X_3)
\end{align*}
and put
\[
\Gamma=\exp\mathbb{Z} X_1\exp\mathbb{Z}X_2\exp\mathbb{Z} X_3.
\]
Then
\[
W\simeq\int_{\mathbb{R}^{\ast}}^{\oplus}\int_{\mathbb{T}}^{\oplus}\text{ }\oplus_{k\in\mathbb{J}}\widetilde{\pi}
_{\lambda,\sigma}|_{G}\text{ }d\lambda\text{ }d\sigma.
\]
Moreover, $\widetilde{\pi}_{\lambda,\sigma}|_{G}$ is a representation of $G$
acting in $L^{2}\left(\mathbb{R}\right)  $ as follows
\begin{align*}
\widetilde{\pi}_{\lambda,\sigma}|_{G}\left(  \exp x_1 X_1\right)  f\left(
t\right)    & =e^{2\pi i\lambda x_1}f\left(  t\right)  \\
\widetilde{\pi}_{\lambda,\sigma}|_{G}\left(  \exp x_2 X_2\right)  f\left(
t\right)    & =e^{-2\pi i\lambda x_2 t}f\left(  t\right)  \\
\widetilde{\pi}_{\lambda,\sigma}|_{G}\left(  \exp x_3X_3\right)  f\left(
t\right)    & =f\left(  t-x_3\right)  \\
\widetilde{\pi}_{\lambda,\sigma}|_{G}\left(  \alpha\right)    & =e^{2\pi
i\sigma}\sqrt{2}f\left(  2t\right)  .
\end{align*}
 Now it is easy to show that the group $\Gamma_\a$ generated by the sets $\a^m(\G)$ is in fact dense in $N$. It follows that the representation $\pi_\l$ restricted to $\Gamma_\a$ is irreducible. Since 
$ \widetilde{\pi}_{\lambda,\sigma}|_{G} $ is an extension of $ \pi |_{\Gamma_\a} $, then it is also irreducible. The above is an irreducible decomposition. We refer to Lemma \ref{ab} and Proposition \ref{irr} for a general proof of this claim.

\
\section{Irreducibility of the Fiber Representations} 

In this section, we would like to obtain some conditions on the irreducibility of the restrictions of irreducible representations of $F$ to $G.$ We recall that a matrix is an \textbf{expansive matrix} if and only if all its eigenvalues have absolute values strictly greater than one. 
\begin{lemma} \label{ab} If $A$ is expansive then $\Gamma_{\alpha}$ is dense in $N$ in the subspace topology of $N.$ 
\end{lemma}
\begin{proof}
Let us assume that $A$ is expansive. It is enough to show that
$$\bigcup_{k\in\mathbb{Z}}A^{k}\left(\mathbb{Z}X_{1}+\cdots+\mathbb{Z}
X_{n}\right)  $$ is dense in $\mathfrak{n}$. Indeed if $c:\mathfrak{n}%
\rightarrow N$ is defined by
\[
c\left({\displaystyle\sum\limits_{i=1}^{n}}
x_{i}X_{i}\right)  =\exp\left(  x_{1}X_{1}\right)  \cdots\exp\left(
x_{n}X_{n}\right)
\]
then $$c\left(  \bigcup_{k\in\mathbb{Z}}A^{k}\left(\mathbb{Z}X_{1}+\cdots+\mathbb{Z}
X_{n}\right)  \right)  =\Gamma_{\alpha}.$$ For each eigenvalue $a_{i}$ of $A,$
$\left\{\mathbb{Z}a_{i}^{k}:k\in\mathbb{Z}\right\}  $ is dense in $\mathbb{R}.$ Thus, given $$x=c\left(  \sum_{i=1}^{n}x_{i}X_{i}\right)  \in N,$$ we have
$l_{1},\cdots,l_{n},m_{1},\cdots,m_{n}\in\mathbb{Z},\text{} m_{i}>0$ such that for every $\varepsilon>0,$ $$\left\vert \dfrac{l_{i}}%
{a_{i}^{m_{i}}}-x_{i}\right\vert <\varepsilon.$$ Put $m=\max_{i}m_{i}$ and let
$j_{i}=l_{i}a_i^{m-m_{i}}.$ Then for each $i,$ $\left\vert \frac{j_{i}}%
{a_{i}^{m}}-x_{i}\right\vert <\varepsilon$ since
\[
\dfrac{l_{i}}{a_{i}^{m_{i}}}=\frac{j_{i}a_{i}^{-m+m_{i}}}{a_{i}^{m_{i}}}%
=\dfrac{j_{i}}{a_{i}^{m}}.
\]
So with $k=-m,$%
\[
\left\Vert A^{k}\left(  j_{1}X_{1}+\cdots+j_{n}X_{n}\right)  -\sum_{i=1}%
^{n}x_{i}X_{i}\right\Vert _{\mathrm{max-norm}}<\varepsilon.
\]
The above norm is the max norm obtained by naturally identifying the Lie algebra $\n$ with the vector space $\R^n.$ 
\end{proof}

\begin{proposition}\label{irr}
If $A$ is expansive then $\pi_{\lambda}|_{\Gamma_{\alpha}}$ is irreducible.
\end{proposition}
\begin{proof}
If $A$ is expansive then $\Gamma_{\alpha}$ is dense in $N$ in the subspace
topology. Now let $f$ be an arbitrary element of $L^{2}\left(\mathbb{R}^{d}\right)  .$ Given a non-zero vector $g,$ for any $\varepsilon>0,$ by the
irreducibility of $\pi_{\lambda}$ there exist $\left\{  x_{k}:1\leq k\leq
m\right\}  \subset N$ and $\left\{  c_{k}:1\leq k\leq m\right\}  \subset\mathbb{C}$ such that
\[
\left\Vert \sum_{k=1}^{m}c_{k}\pi_{\lambda}\left(  x_{k}\right)
g-f\right\Vert _{L^{2}\left(\mathbb{R}^{d}\right)  }<\frac{\varepsilon}{2}.
\]
Now, since $\Gamma_{\alpha}$ is dense in $N$ and since $\pi_{\lambda}$ is a
continuous representation, then there exists $\left\{  \gamma_{k}:1\leq k\leq
m\right\}  $ $\subset$ $\Gamma_{\alpha}$ such that
\[
\left\Vert \pi_{\lambda}\left(  x_{k}\right)  g-\pi_{\lambda}\left(
\gamma_{k}\right)  g\right\Vert <\frac{\varepsilon}{2m\left\vert
c_{k}\right\vert }.
\]
Now,
\begin{align*}
& \left\Vert \sum_{k=1}^{m}c_{k}\pi_{\lambda}\left(  \gamma_{k}\right)
g-f\right\Vert _{L^{2}\left(\mathbb{R}^{d}\right)  }\\
& =\left\Vert \sum_{k=1}^{m}c_{k}\pi_{\lambda}\left(  \gamma_{k}\right)
g-\sum_{k=1}^{m}c_{k}\pi_{\lambda}\left(  x_{k}\right)  g+\sum_{k=1}^{m}%
c_{k}\pi_{\lambda}\left(  x_{k}\right)  g-f\right\Vert _{L^{2}\left(\mathbb{R}^{d}\right)  }\\
& \leq\left\Vert \sum_{k=1}^{m}c_{k}\pi_{\lambda}\left(  \gamma_{k}\right)
g-\sum_{k=1}^{m}c_{k}\pi_{\lambda}\left(  x_{k}\right)  g\right\Vert
_{L^{2}\left(\mathbb{R}^{d}\right)  }+\left\Vert \sum_{k=1}^{m}c_{k}\pi_{\lambda}\left(
x_{k}\right)  g-f\right\Vert _{L^{2}\left(\mathbb{R}^{d}\right)  }\\
& \leq\sum_{k=1}^{m}\left\vert c_{k}\right\vert \left\Vert \pi_{\lambda
}\left(  \gamma_{k}\right)  g-\pi_{\lambda}\left(  x_{k}\right)  g\right\Vert
_{L^{2}\left(\mathbb{R}^{d}\right)  }+\left\Vert \sum_{k=1}^{m}c_{k}\pi_{\lambda}\left(
x_{k}\right)  g-f\right\Vert _{L^{2}\left(\mathbb{R}^{d}\right)  }\\
& \leq\sum_{k=1}^{m}\left\vert c_{k}\right\vert \frac{\varepsilon
}{2m\left\vert c_{k}\right\vert }+\frac{\varepsilon}{2}\\
& \leq\frac{\varepsilon}{2}+\frac{\varepsilon}{2}=\varepsilon.
\end{align*}
Thus, the linear span of $\pi_{\lambda}\left(  \Gamma_{\alpha}\right)  g$ is
dense in $L^{2}\left(\mathbb{R}^{d}\right)  $, and it follows that $\pi_{\lambda}|_{\Gamma_{\alpha}}$ is
irreducible. 
\end{proof}
\begin{example}\label{example}
\end{example}\noindent Let $N$ be a nilpotent Lie group with Lie algebra spanned by $\left\{
Z_{1},Z_{2},X_{1},X_{2},Y\right\}  $ such that the only non-trivial Lie
brackets are $\left[  X_{i},Y\right]  =Z_{i}.$ Now, let $\mathbf{e}=\left\{
3,5\right\}  .$ Then
\[
\Omega=\left\{
\begin{array}
[c]{c}%
\lambda_{1}Z_{1}^{\ast}+\lambda_{2}Z_{2}^{\ast}+\beta_{1}X_{1}^{\ast}%
+\beta_{2}X_{2}^{\ast}+\gamma Y^{\ast}\in\mathfrak{n}^{\ast}:\lambda_{1}%
\neq0\\
\left(  \lambda_{1},\lambda_{2},\beta_{1},\beta_{2},\gamma\right)  \in%
\mathbb{R}
^{5}%
\end{array}
\right\}
\]
and the unitary dual of $N$ is parametrized by
\[
\Lambda=\left\{  \lambda_{1}Z_{1}^{\ast}+\lambda_{2}Z_{2}^{\ast}+\beta
_{2}X_{2}^{\ast}\in\Omega:\left(  \lambda_{1},\lambda_{2},\beta_{2}\right)
\in%
\mathbb{R}
^{3}\right\}  \simeq%
\mathbb{R}
^{\ast}\times%
\mathbb{R}
^{2}.
\]
Now for each $\lambda\in\Lambda,$ $\pi_{\lambda}$ is realized as acting in
$L^{2}\left(
\mathbb{R}
\right)  $ as follows:%
\begin{align*}
\pi_{\lambda}\left(  \exp x_{1}X_{1}\right)  f\left(  y\right)    & =e^{2\pi
ix_{1}y\lambda_{1}}f\left(  y\right)  \\
\pi_{\lambda}\left(  \exp x_{2}X_{2}\right)  f\left(  y\right)    & =e^{2\pi
ix_{2}\left(  \beta_{2}+y\lambda_{2}\right)  }f\left(  y\right)  \\
\pi_{\lambda}\left(  \exp sY\right)  f\left(  y\right)    & =f\left(
y-s\right)  \\
\pi_{\lambda}\left(  \exp z_{k}Z_{k}\right)  f\left(  y\right)    & =e^{2\pi
iz_{k}\lambda_{k}}f\left(  y\right)  .
\end{align*}
Define
\[
\Gamma=\exp\left(
\mathbb{Z}
Z_{1}\right)  \exp\left(
\mathbb{Z}
Z_{2}\right)  \exp\left(
\mathbb{Z}
X_{1}\right)  \exp\left(
\mathbb{Z}
X_{2}\right)  \exp\left(
\mathbb{Z}
Y\right)  .
\]
Now, let $\alpha\in\mathrm{Aut}\left(  N\right)  \ $such that
\begin{align*}
& \alpha\left(  \exp\left(  z_{1}Z_{1}\right)  \exp\left(  z_{2}Z_{2}\right)
\exp\left(  x_{1}X_{1}\right)  \exp\left(  x_{2}X_{2}\right)  \exp\left(
yY\right)  \right)  \\
& =\exp\left(  4z_{1}Z_{1}\right)  \exp\left(  4z_{2}Z_{2}\right)  \exp\left(
2x_{1}X_{1}\right)  \exp\left(  2x_{2}X_{2}\right)  \exp\left(  2yY\right)  .
\end{align*}
Then
\[
W\simeq\int_{E}^{\oplus}\oplus_{k\in\mathbb{I}}\text{ }\mathrm{Ind}%
_{\Gamma_{\alpha}}^{\Gamma_{\alpha}\rtimes H}\left(  \pi_{\lambda}%
|_{\Gamma_{\alpha}}\right)  \text{ }d\lambda
\]
and $\mathrm{Ind}_{\Gamma_{\alpha}}^{\Gamma_{\alpha}\rtimes H}\left(
\pi_{\lambda}|_{\Gamma_{\alpha}}\right)  $ is irreducible for each $\lambda
\in E.$
\vskip 0.5cm

\subsection{Bekka-Driutti Condition for Irreducibility}
In this subsection, we will recall a result of Bekka and Driutti (see \cite{Bekka}.) Let $$\left\{  X_{1},\cdots
,X_{m},\cdots,X_{n}\right\}  $$ be a strong Malcev basis of $\mathfrak{n}$
strongly based on $\Gamma$ passing through $\left[  \mathfrak{n}%
,\mathfrak{n}\right]  $ such that $\dim\left[  \mathfrak{n},\mathfrak{n}%
\right]  =m.$  Let
$\widetilde{p}:\mathfrak{n}\rightarrow\left[  \mathfrak{n},\mathfrak{n}%
\right]  $ be the canonical projection. A subalgebra $\mathfrak{h}$ is not contained in a proper
rational ideal of $\mathfrak{n}$ if and only if $\widetilde{p}\left(
\mathfrak{h}\right)  $ is not contained in a proper rational subspace of
$\mathfrak{n/}\left[  \mathfrak{n},\mathfrak{n}\right]  .$ In fact, this is
the case if and only if there exists $$X=\sum_{i=1}^{n}x_{i}X_{i}%
\in\mathfrak{h}$$ such that $\dim_{%
\mathbb{Q}
}\left\{  x_{m+1},\cdots,x_{n}\right\}  =n-m.$ Moreover, the restriction
$\pi_{\lambda}|_{\Gamma}$ of $\pi_{\lambda}$ to $\Gamma$ is irreducible if and
only if the radical (see \ref{radical}) $\mathfrak{n}\left(  \lambda\right)  $ is not contained in
a proper rational ideal of $\mathfrak{n}$. Since $\Gamma
\subseteq\Gamma_{\alpha},$ then the following holds. 
\begin{proposition}
 If the radical $\mathfrak{n}\left(  \lambda\right)  $ is
not contained in a proper rational ideal of $\mathfrak{n}$ then $\pi_{\lambda
}|_{\Gamma_{\alpha}}$ is irreducible.
\end{proposition}

\begin{example}
\end{example}\noindent 
Let $N$ be the freely generated two step nilpotent Lie group with three
generators. The Lie algebra of $N$ is spanned by $
\left\{  Z_{23},Z_{13},Z_{12},Z_{3},Z_{2},Z_{1}\right\}$ such that the only non-trivial Lie brackets are defined as follows: $
\left[  Z_{i},Z_{j}\right]  =Z_{ij}\text{ for }i<j.$ Now let
\[
\Gamma=\exp\mathbb{Z}Z_{23}\exp\mathbb{Z}Z_{13}\exp\mathbb{Z}
Z_{12}\exp\mathbb{Z}Z_{3}\exp\mathbb{Z}Z_{2}\exp\mathbb{Z}Z_{1}.
\]
Thus, the basis above is a strong Malcev basis strongly based on $\Gamma$ and is
passing through the ideal $
\left[  \mathfrak{n,n}\right]  =\mathfrak{z.}$ It is not too hard to see that the radical corresponding to $\lambda$ is:
\[
\mathfrak{n}\left(  \lambda\right)  =\mathfrak{z\oplus}\text{ }\mathbb{R}
\left(  \lambda\left(  Z_{23}\right)  Z_{1}-\lambda\left(  Z_{13}\right)
Z_{2}+\lambda\left(  Z_{12}\right)  Z_{3}\right)  .
\]
It is clear that $\mathfrak{n}\left(  \lambda\right)  $ is not
contained in a proper rational ideal if
\[
\dim_{\mathbb{Q}}\left(  \lambda\left(  Z_{23}\right)  ,-\lambda\left(  Z_{13}\right)
,\lambda\left(  Z_{12}\right)  \right)  =3.
\]
Next, define $\alpha$ such that
\[
\alpha\left(  \exp Z_{k}\right)  =\left\{
\begin{array}
[c]{cc}%
\exp\left(  2Z_{k}\right)   & \text{ if }k=23\\
\exp\left(  2Z_{k}\right)   & \text{if }k=13\\
\exp\left(  Z_{k}\right)   & \text{if }k=12\\
\exp\left(  2Z_{k}\right)   & \text{if }k=3\\
\exp\left(  Z_{k}\right)   & \text{if }k=2\\
\exp\left(  Z_{k}\right)   & \text{if }k=1
\end{array}
\right.  .
\]
Then
\[
W\simeq\int_{E}^{\oplus}\oplus_{k\in\mathbb{I}}\text{ \textrm{Ind}}%
_{\Gamma_{\alpha}}^{\Gamma_{\alpha}\rtimes H}\left(  \pi_{\lambda}%
|_{\Gamma_{\alpha}}\right)  \text{ }d\lambda
\]
where
\[
E=\left\{  A^t\left(  \pm Z_{23}^{\ast}+\beta_{1}Z_{13}^{\ast}+\beta_{2}%
Z_{12}^{\ast}+\beta_{3}Z_{1}^{\ast}\right)  :t\in\left[  0,1\right)
,\beta_{k}\in\mathbb{R}\right\}
\]
and almost every representation $\pi_{\lambda}|_{\Gamma_{\alpha}}$
is irreducible. 
\section*{Acknowledgment} Brad Currey and Vignon Oussa would like to thank the Graduate Center of the City University of New York for their hospitality. We also thank the referee for providing constructive comments and help in improving the paper.

\end{document}